\documentclass[a4paper,12pt, reqn0]{amsart}
\pdfoutput=1
\usepackage{dsfont}
\usepackage{bbm}
\usepackage[nodate]{datetime}
\usepackage[hmargin=3cm]{geometry}
\usepackage{color}
\usepackage{subfig}
\usepackage{fancyhdr}
\usepackage{amsmath,amssymb,amsfonts,amsthm}
\usepackage{mathtools, bm}
\usepackage{amstext}
\usepackage{amsmath}
\usepackage{amssymb}
\usepackage{enumerate}
\usepackage{amsbsy}
\usepackage{amsopn}
\usepackage{bbm,amsthm}
\usepackage{amscd}
\usepackage{amsxtra}
\usepackage{todonotes}

\newtheorem{theorem}{Theorem}[section]
\newtheorem*{theorem*}{Theorem}
\newtheorem{lemma}{Lemma}[section]
\newtheorem{proposition}{Proposition}[section]
\newtheorem{corollary}{Corollary}[section]

\theoremstyle{remark}

\newtheorem{example}{Example}[section]
\newtheorem{remark}{Remark}[section]

\newcommand{\CC}{\mathds{C}}

\newcommand{\NN}{\mathds{N}}

\newcommand{\ind}{\mathds{1}}

\newcommand{\bt}{\beta}

\newcommand{\U}{\mathcal{U}}
\newcommand{\Mcal}{\mathcal{M}}
\newcommand{\A}{\mathcal{A}}

\DeclareMathOperator{\lcm}{lcm}

\title{Convolution of periodic multiplicative functions and the divisor problem}
\author{Marco Aymone}
\address{Departamento de Matem\'atica, Universidade Federal de Minas Gerais, Av. Ant\^onio Carlos, 6627, CEP 31270-901, Belo Horizonte, MG, Brazil.}
\email{aymone.marco@gmail.com}
\author{Gopal Maiti}
\address{ Max-Planck Institute for Mathematics, Vivatsgasse 7, Bonn 53111, Germany. }
\email{maiti@mpim-bonn.mpg.de}
\author{Olivier Ramar\'e}
\address{CNRS /  Institut de Math\'ematiques de Marseille,
Aix Marseille Universit\'e, U.M.R. 7373,
Campus de Luminy, Case 907,
13288 MARSEILLE Cedex 9, France }
\email{olivier.ramare@univ-amu.fr}
\author{Priyamvad Srivastav}
\email{priyamvads@gmail.com}

\date{\today}
\subjclass[2010]{}

\begin{document}
\begin{abstract} 
We study a certain class of arithmetic functions that appeared in Klurman's classification of $\pm 1$ multiplicative functions with bounded partial sums, c.f., Comp. Math. 153 (8), 2017, pp. 1622-1657. These functions are periodic and $1$-pretentious. We
prove that if $f_1$ and $f_2$ belong to this class, then $\sum_{n\leq x}(f_1\ast f_2)(n)=\Omega(x^{1/4})$. This confirms a conjecture by the first author. As a byproduct of our proof, we studied the correlation between $\Delta(x)$ and $\Delta(\theta x)$, where $\theta$ is a fixed real number. We prove that there is a non-trivial correlation when
$\theta$ is rational, and a decorrelation when $\theta$ is irrational. Moreover,
if $\theta$ has a finite irrationality measure, then we can make it
quantitative this decorrelation in terms of this measure.
\end{abstract}

\maketitle
\section{Introduction}
\subsection{Main result and background} A question posed by Erd\H{o}s
in~\cite{erdos_unsolved}, known as the Erd\H{o}s discrepancy problem,
states that whether for all arithmetic functions $f:\NN\to\{-1,1\}$ we
have that the discrepancy
\begin{equation}\label{equation discrepancy}
\sup_{x,d}\left|\sum_{n\leq x}f(nd)\right|=\infty.
\end{equation}
When in addition $f$ is assumed to be completely multiplicative, then
this reduces to whether $f$ has unbounded partial sums.

In 2015, Tao \cite{tao_discrepancy} proved that \eqref{equation
  discrepancy} holds for all $f:\NN\to\{-1,1\}$, and a key point of
its proof is that it is sufficient to establish \eqref{equation
  discrepancy} only in the class of completely multiplicative
functions $f$ taking values in the unit (complex) circle.

When $f:\NN\to\{-1,1\}$ is assumed to be only multiplicative, then not
necessarily $f$ has unbounded partial sums. For example,
$f(n)=(-1)^{n+1}$ is multiplicative and clearly has bounded partial
sums. In this case, $f(2^k)=-1$ for all positive integers $k$. It was
observed by Coons \cite{coons_erdos} that, for bounded partial sums, this rigidity on powers of
$2$ is actually necessary under suitable conditions on the values that a multiplicative function
$f$ takes at the remaining primes. Later, in the same paper
\cite{tao_discrepancy}, Tao gave a partial classification of
multiplicative functions taking values $\pm 1$ with bounded partial
sums: They must satisfy the previous rigidity condition on powers of
$2$ and they must be $1$-pretentious (for more on pretentious Number
Theory we refer the reader to \cite{granvillepretentious} by Granville and
Soundararajan), that is,
$$\sum_{p}\frac{1-f(p)}{p}<\infty.$$

Later, Klurman \cite{Klurman*17} proved that the only multiplicative
functions $f$ taking $\pm 1$ values and with bounded partial sums are
the periodic multiplicative functions with sum $0$ inside each period,
and thus, closing this problem for $\pm1$ multiplicative functions.

Building upon the referred work of Klurman, the first author proved
in~\cite{Aymone*22} that if we allow values outside the unit disk, a
$M$-periodic multiplicative function $f$ with bounded partial sums
such that $f(M)\neq 0$ satisfies
\begin{enumerate}[i.]
\item For some prime $q|M$, $\sum_{k=0}^\infty \frac{f(q^k)}{q^k}=0$.
\item For each $p^a\| M$, $f(p^k)=f(p^a)$ for all $k\geq a$.
\item For each $\gcd(p,M)=1$, $f(p^k)=1$, for all $k\geq 1$.
\end{enumerate}
\noindent Conversely, if $f:\NN\to\CC$ is multiplicative and the three conditions above are satisfied, then $f$ has period $M$ and has bounded partial sums. Therefore, these three conditions above give examples of multiplicative functions with values outside the unit disk with bounded partial sums, despite of the fact that $f(M)$ is zero or not.

\begin{remark} It is interesting to observe that when it is assumed that $|f|\leq 1$, the only way to achieve condition i. is with $q=2$ and $f(2^k)=-1$ for all $k\geq 1$.
\end{remark}

\begin{remark} What makes the difference between a multiplicative
  function $f$ satisfying i-ii-iii from a non-principal Dirichlet
  character $\chi$ is that $\chi$ neither satisfies i. nor iii.
\end{remark}

Here we are interested in the convolution $f_1\ast f_2(n):=\sum_{d|n}f_1(d)f_2(n/d)$ for $f_1$ and
$f_2$ satisfying i-ii-iii above. It was proved in~\cite{Aymone*22}
that
\begin{equation*}
  \sum_{n\leq x}(f_1\ast f_2)(n)\ll x^{\alpha+\epsilon},
\end{equation*}
where $\alpha$ is the infimum over the exponents $a>0$ such that
$\Delta(x)\ll x^{a}$, where $\Delta(x)$ is the classical error term in
the Dirichlet divisor problem defined by:
\begin{equation*}
  \sum_{n\leq x}\tau(n)=x\log x +(2\gamma-1)x+\Delta(x).
\end{equation*}

It is widely believed that $\alpha=1/4$ and many results were proven in this direction. The best upper bound up to date is due to Huxley \cite{Huxley2003}: $\alpha\leq 131/416 \approx 0.315$. Regarding $\Omega$ bounds, Soundararajan \cite{sound_omega_divisor} proved that
$$\Delta(x)=\Omega\left((x\log x)^{1/4}\frac{(\log\log x)^{3/4(2^{4/3}-1)}}{(\log\log\log  x)^{5/8}}\right).$$

It was conjectured in \cite{Aymone*22} that the partial sums of
$f_1\ast f_2$ obey a similar $\Omega$-bound for $\Delta(x)$, that is,
$\sum_{n\leq x}(f_1\ast f_2)(n)=\Omega(x^{1/4})$. Here we establish
this conjecture.
\begin{theorem}
  \label{theorem principal}
  Let $f_1$ and $f_2$ be periodic multiplicative functions satisfying i-ii-iii above. Then
  $
    \sum_{n\leq x}(f_1\ast f_2)(n)=\Omega(x^{1/4})$.
\end{theorem}

\begin{example} The results from \cite{Aymone*22} give that for each
  prime $q$ there exists a unique $q$-periodic multiplicative function
  $f$ with bounded partial sums and such that $f(q)\neq 0$. In the
  case $q=2$, the corresponding function is
  $f(n)=(-1)^{n+1}$. Therefore, in this particular case we have that
  $\sum_{n\leq x}(f\ast f)(n)=\Omega(x^{1/4})$. In particular, this
  establishes the conjecture in an uncovered case by Proposition 3.1
  of \cite{Aymone*22}.
\end{example}

\begin{remark} Another class of periodic multiplicative functions with bounded partial sums is that of the non-principal Dirichlet characters. In a forthcoming work, the first author is finishing a study where he shows a similar $\Omega$-bound for the partial sums of the convolution between these Dirichlet characters.
\end{remark}

Our proof relies on two ingredients. The second one is a study of a
family of quadratic forms, and is explained in Section~\ref{section quadratic}.
The first ingredient is a generalization
of a result of Tong~\cite{Tong*56} and prove the next theorem.
\begin{theorem}
  \label{TongGen}
  When $a$ and $b$ are non negative integers, $\lambda=\gcd(a,b)$, $c=a/\lambda$ and
  $d=b/\lambda$, we have
  \begin{equation*}
    \lim_{X\to\infty}\frac{1}{X^{3/2}}\int_{1}^X
    \Delta(x/a)\Delta(x/b)dx=\frac{\tau(cd)}{6\pi^2\sqrt{\lambda}cd}
    \frac{\zeta(3/2)^4}{\zeta(3)}
    \prod_{p^k\|cd}\frac{1-\frac{k-1}{(k+1)p^{3/2}}}{1+1/p^{3/2}}.
  \end{equation*}
  Furthermore, when $\theta>0$ is irrational, we have
  \begin{equation*}
    \lim_{X\to\infty}\frac{1}{X^{3/2}}\int_{1}^X
    \Delta(x)\Delta(\theta x)dx=0.
  \end{equation*}
\end{theorem}
\subsection{The proof in the large}
To prove Theorem \ref{theorem principal}, our starting point is the following formula from \cite{Aymone*22}: If $M_1$ and $M_2$ are the periods of $f_1$ and $f_2$ respectively, then
\begin{equation}\label{equation sum f ast g}
  \sum_{n\leq x}(f_1\ast f_2)(n)=\sum_{n|M_1M_2}(f_1\ast f_2\ast \mu \ast \mu)(n)\Delta(x/n),
\end{equation}
where $\mu$ is the M\"obius function. Therefore, the partial sums of
$f_1\ast f_2$ can be written as a finite linear combination of the
quantities $(\Delta(x/n))_n$. Apart from the fact that~$\Delta(x)=\Omega(x^{1/4})$, we cannot, at least by a direct argument,
prevent a conspiracy among the large values of $(\Delta(x/n))_n$ that
would yield a cancellation among a linear combination
of them.

To circumvent this, our approach is inspired by an elegant result of Tong \cite{Tong*56}:
\begin{equation}\label{Tong}
  \int_{1}^X\Delta(x)^2 dx=\frac{(1+o(1))}{6\pi^2}\sum_{n=1}^\infty
    \frac{\tau(n)^2}{n^{3/2}}\,X^{3/2}.
\end{equation}
By \eqref{equation sum f ast g}, the limit
$$\lim_{X\to\infty}\frac{1}{X^{3/2}}\int_{1}^X\left|\sum_{n\leq x}(f_1\ast f_2)(n)\right|^2dx$$
can be expressed as a quadratic form with matrix $(c_{a,b})_{a,b|M_1M_2}$ where $c_{a,b}$ is the correlation
\begin{equation*}
  c_{a,b}:=\lim_{X\to\infty}\frac{1}{X^{3/2}}\int_{1}^X
  \Delta(x/a)\Delta(x/b)dx.
\end{equation*}

As it turns out, these correlations do not vanish and are computed in Theorem~\ref{TongGen}.
With that in hand, the matrix correlation-term $c_{a,b}$ can be
expressed as
\begin{equation}\label{equation matrix form}
\frac{C}{\sqrt{\gcd(a,b)}}\varphi\left(\frac{\lcm(a,b)}{\gcd(a,b)}\right),
\end{equation}
for some constant $C>0$ and multiplicative function $\varphi$. 

This matrix entanglement is hard to analyze directly. In
Section~\ref{section quadratic} we explore sufficient conditions for a
matrix of the form \eqref{equation matrix form} to be positive
definite. When this happens, this ensures the referred
$\Omega$-bound. Thanks to the Selberg diagonalization process, we
show that when $\varphi$ is completely multiplicative and satisfies
other conditions, then this matrix is positive definite.
The main proof somehow reduces to this case;
we indeed find a way to conjugate our original matrix to reach
 a matrix related to a completely multiplicative
function. With standard linear algebra of Hermitian matrices we
conclude that our matrix $(c_{a,b})_{a,b|M_1M_2}$ is positive
definite. We ended up with the following result.

\begin{theorem} Let $f_1$ and $f_2$ be two periodic multiplicative functions satisfying i-ii-iii above with periods $M_1$ and $M_2$ respectively. Let $g=f_1\ast f_2\ast \mu\ast\mu$. Then the following limit 
$$\lim_{X\to\infty}\frac{1}{X^{3/2}}\int_1^X\left|\sum_{n\leq x}(f_1\ast f_2)(n)  \right|^2 dx$$
is positive and equals to
$$\frac{\zeta(3/2)^4}{\zeta(3)}\sum_{n,m|M_1M_2}g(n)\overline{g(m)}\frac{\gcd(n,m)^{3/2}}{nm}\tau\left(\frac{nm}{\gcd(n,m)^2}\right)
    \prod_{p^{k}\| nm/\gcd(n,m)^2 }  \frac{1-\tfrac{(k-1)}{(k+1)} p^{-3/2} }{1+ p^{-3/2}}.$$

\end{theorem}

\subsection{Byproduct study} Motivated by Nyman's reformulation of the
Riemann hypothesis \cite{nyman_RH}, in recent papers
\cite{duarte_balazard,balazard_martin_hal,balazard_martin} by
Balazard, Duarte and Martin, the correlation
\begin{equation*}
  A(\theta):=\int_0^\infty\{x\}\{\theta x\}\frac{dx}{x^2}
\end{equation*}
has been thoroughly studied. Here
$\theta>0$ is any real number and $\{x\}$ stands for the
fractional part of $x$. Several analytic properties for the function
$A(\theta)$ have been shown.

Motivated by this, we studied the ``\textit{divisor}'' analogue 
$$I(\theta)=\lim_{X\to\infty}\frac{1}{X^{3/2}}\int_{1}^X \Delta(x)\Delta(\theta x)dx.$$
As stated in Theorem~\ref{TongGen}, when $\theta=p/q$ is a rational
number, the limit above is described by a positive multiplicative
function depending on $p$ and $q$. However and somewhat surprisingly, when
$\theta$ is irrational, this correlation vanishes. The next proposition
establishes that this vanishing is indeed very strong, except maybe at
points $\theta$ that are well approximated by rationals.

\begin{proposition}\label{proposition decorrelation} Let $\theta>0$ be
  an irrational number with irrationality measure $\eta+1$, that is,
  for each $\epsilon>0$ there is a constant $C>0$ such that the
  inequality
$$|n-m\theta|\geq \frac{C}{m^{\eta+\epsilon}}$$
is violated only for a finite number of positive integers $n$ and
$m$. Then, for every positive $\epsilon$, we have
$$\int_{1}^X\Delta(x)\Delta(\theta x)dx=O(X^{3/2-1/(18\eta)+\epsilon}).$$
In the other cases of irrationals $\theta$, the integral above is $o(X^{3/2})$.
\end{proposition}
This shows that we have decorrelation among the values $\Delta(x)$ and
$\Delta(\theta x)$ when $\theta$ is irrational, and moreover, this
gives that the function $I(\theta)$ is continuous at the irrational numbers and discontinuous at the rationals. Another interesting remark is a result due to Khintchine \cite{khinchin_1924} that states that almost all irrational numbers, with respect to (w.r.t.) Lebesgue measure, have irrationality measure equals to $2$.

Therefore, this result of Khintchine allow us to state the following Corollary from Proposition \ref{proposition decorrelation}.

\begin{corollary} For almost all irrational numbers $\theta$ w.r.t. Lebesgue measure, for all small fixed $\epsilon>0$,
$$\int_{1}^X\Delta(x)\Delta(\theta x)dx=O(X^{3/2-1/18+\epsilon}).$$
\end{corollary}

We mention that a similar decorrelation also has been
obtained by Ivi\'c and Zhai in \cite{ivic}. In this paper they show
decorrelation between $\Delta(x)$ and $\Delta_k(x)$, where
$\Delta_k(x)$ is the error term related to the $k$-fold divisor
function, and $k=3$ or $4$.

\section{Notation}
\subsection{Asymptotic notation} We employ both Vinogradov's notation
$f\ll g$ or $f=O(g)$ whenever there exists a constant $C>0$ such that
$|f(x)|\leq C|g(x)|$, for all~$x$ in a set of parameters. When not
specified, this set of parameters is~$x\in (a,\infty)$ for
sufficiently large $a>0$. We employ $f=o(g)$ when
$\lim_{x\to a} \frac{f(x)}{g(x)}=0$. In this case $a$ can be a complex
number or $\pm\infty$. Finally, $f=\Omega(g)$ when
$\limsup_{x\to a}\frac{|f(x)|}{g(x)}>0$, where $a$ is as in the
previous notation.
\subsection{Number-theoretic notation} Here $p$ stands for a generic
prime number. We sometimes denote the least common multiple between
$a,b$ as $\lcm(a,b)$. The greatest common divisor is denoted by
$\gcd(a,b)$. The symbol $\ast$ stands for Dirichlet convolution between two arithmetic functions: $(f\ast g)(n)=\sum_{d|n}f(d)g(n/d)$.

\section{Multiplicative auxiliaries}
Our first task is to evaluate
$
  \sum_{n\ge1}{\tau(cn)\tau(dn)}/{n^{3/2}}
$
for coprime positive integers~$c$ and~$d$. 
\begin{lemma}\label{multiplicativity}
  Let $c$ be fixed positive number and $f(n)$ be a multiplicative
  function with $f(c)\neq 0$. Then $n\mapsto\frac{f(cn)}{f(c)}$ is
  multiplicative.
\end{lemma}
\begin{proof}
  For positive integers $u,v$, we
  have \begin{equation*}
    f(u)f(v)=f(\gcd(u,v))f(\lcm(u,v)).
  \end{equation*}
  Let $u=cn$, $v=cm$ with $\gcd(n,m)=1$. Then $f(cn)f(cm)=f(c)f(cnm)$.
  Therefore, we obtain
  \begin{equation*}
    \frac{f(cm)}{f(c)}\frac{f(cn)}{f(c)}=\frac{f(cnm)}{f(c)}.
  \end{equation*}
\end{proof}

\begin{lemma}
  \label{exactcomp}
  Let $c,d$ be two fixed positive integers with $\gcd(c,d)=1$. Then 
  \begin{align*}
    \sum_{n= 1}^\infty \frac{\tau(cn)\tau(dn)}{n^{s}}
    =\tau(cd)\frac{\zeta(s)^4}{\zeta(2s)}
    \prod_{p^{k}\| cd } \left(1+ p^{-s}\right)^{-1} \left(1-\tfrac{(k-1)}{(k+1)} p^{-s}\right).
  \end{align*}
\end{lemma}
The quantity we compute appears in several places, for instance in
\cite{Lee-Lee*19} by Lee and Lee and in \cite{Munsch-Shparlinski*23} by Borda, Munsch and Shparlinski.
\begin{proof}
  Note that $\frac{\tau(cn)}{\tau(c)}$ is a multiplicative function in the variable $n$ by Lemma \ref{multiplicativity}, and so is
  $\frac{\tau(cn)\tau(dn)}{\tau(c)\tau(d)}$.  Therefore, for
  $\Re(s)>1$ we have the following Euler factorization
  \begin{align*}
    \sum_{n= 1}^\infty \frac{\tau(cn)\tau(dn)} {\tau(c)\tau(d)n^{s}}
    =
    \prod_{p\nmid cd }
    \left(1 + \sum_{\ell= 1}^\infty \frac{\tau(p^\ell)^{2}}{p^{\ell s}} \right)
    \prod_{p\mid cd }
    \left(1 + \sum_{\ell= 1}^\infty
    \frac{\tau(cp^\ell)\tau(dp^\ell)}{\tau(c) \tau(d) p^{\ell s}} \right).
   \end{align*}

For $|x|<1$, we know that
\begin{equation*}
  \sum_{\ell= 0}^\infty(\ell+1)x^{\ell}
  = \frac{1}{(1-x)^{2}},
  \;\;\quad\quad
  \sum_{\ell=0}^\infty(\ell+1)^{2}x^{\ell}
  = \frac{(1+x)}{(1-x)^{3}},
\end{equation*}
from which we also derive that
\begin{equation*}
  \sum_{\ell=0}^\infty\ell(\ell+1) x^{\ell}= \frac{2x}{(1-x)^{3}}.
\end{equation*}
Now,
\begin{align*}
  \prod_{p\nmid cd } \left(1 + \sum_{\ell=
  1}^\infty\frac{\tau(p^\ell)^{2}}{p^{\ell s}} \right)
  &= \prod_{p } \left(1 + \sum_{\ell= 1}^\infty\frac{(\ell+1)^{2}}{p^{\ell
    s}} \right)
    \prod_{p\mid cd } \left(1 + \sum_{\ell=
    1}^\infty\frac{(\ell+1)^{2}}{p^{\ell s}} \right)^{-1}\\
  &=\prod_{p } \frac{\left(1+ p^{-s}\right)}{\left(1- p^{-s}\right)^{3}}\prod_{p| cd } \frac{\left(1- p^{-s}\right)^{3}}{\left(1+ p^{-s}\right)}\\
  &=\frac{\zeta(s)^4}{\zeta(2s)} \prod_{p| cd } \frac{\left(1- p^{-s}\right)^{3}}{\left(1+ p^{-s}\right)}.
\end{align*}	
If $\gcd(c,d)=1$
\begin{align*} 
  \prod_{p\mid cd } \left(1 + \sum_{\ell= 1}^\infty
  \frac{\tau(cp^\ell)\tau(dp^\ell)}{\tau(c) \tau(d) p^{\ell s}} \right)
  &=\prod_{p^{k}\| cd }
    \left(1 + \sum_{\ell= 1}^\infty\frac{(k+1+\ell)(\ell+1)}{(k+1)p^{\ell s}} \right) \\
  &=\prod_{p^{k}\| cd }
    \left( 1+ \sum_{\ell= 1}^\infty\frac{(\ell+1)}{p^{\ell s}}
    +\frac{1}{k+1}\sum_{\ell= 1}^\infty\frac{\ell(\ell+1)}{p^{\ell s}} \right)\\
  &=\prod_{p^{k}\| cd } \left(1- p^{-s}\right)^{-3}\left(1-\tfrac{(k-1)}{(k+1)} p^{-s}\right).
\end{align*}
	
\end{proof}	
\section{Correlations of the $\Delta$ function}

We continue with the proof with the following Lemma.

\begin{lemma}\label{lema integrais} Let $a>0$. Then
$$\int x^2\cos(ax)dx=x^2\frac{\sin(ax)}{a}+2x\frac{\cos(ax)}{a^2}-2\frac{\sin(ax)}{a^3}.$$
Moreover, for any $X>1$,
$$\int_{1}^Xx^2\cos(ax)dx\ll \frac{X^2}{a},\;\int_{1}^Xx^2\sin(ax)dx\ll \frac{X^2}{a}.$$
\end{lemma}
\begin{proof}
We do integration by parts:
$$\int x^2\cos(ax)dx=x^2\frac{\sin(ax)}{a}-\int \frac{2x\sin(ax)}{a}dx.$$
By making the trivial bound $|\sin(ax)|\leq 1$ in the right hand side of the equation above, we reach to the second claim of the proposed Lemma. By making integration by parts in the last integral of the equation above gives the first claim of the Lemma. Similar arguments give similar results for $\sin$ in place of $\cos$. \end{proof}

\begin{lemma}
\label{firstlimit}
  Let $a,b$ be positive integers,
  $\lambda=\gcd(a,b)$, $c=a/\lambda$ and $d=b/\lambda$. Then
  \begin{equation*}
    \lim_{X\to\infty}\frac{1}{X^{3/2}}\int_1^X\Delta(x/a)\Delta(x/b)dx
    =
    \frac{1}{6\pi^2\sqrt{\lambda}cd}
    \sum_{n=1}^\infty\frac{\tau(cn)\tau(dn)}{n^{3/2}}.
  \end{equation*}
\end{lemma}
\begin{proof}
  Let $N>0$ and $\epsilon>0$ be a small number that may change from
  line after line. We proceed with Vorono\"i's formula for $\Delta(x)$
  in the following form (see \cite{Lau-Tsang*95})
\begin{equation*}
    \Delta(x)=\frac{x^{1/4}}{\pi\sqrt{2}}\sum_{n\leq N}\frac{\tau(n)}{n^{3/4}}\cos(4\pi\sqrt{nx}-\pi/4)+R_N(x)
\end{equation*}
where, for every positive $\epsilon$, we have
\begin{equation*}
R_N(x)\ll x^\epsilon+\frac{x^{1/2+\epsilon}}{N^{1/2}}.
\end{equation*}

We select $N$ at the end. With this formula we have that in the range $1\leq x\leq X$,
\begin{equation*}
  \Delta(x/a)
  =\frac{(x/a)^{1/4}}{\pi\sqrt{2}}
  \sum_{n\leq
    N}\frac{\tau(n)}{n^{3/4}}\cos(4\pi\sqrt{nx/a}-\pi/4)
  +R_N(x/a)=U_N(x/a)+R_N(x/a)
\end{equation*}
say.

Now, 
\begin{align*}
\int_{1}^X\Delta(x/a)\Delta(x/b)dx&=\int_{1}^X U_N(x/a)U_N(x/b)dx+\int_{1}^X U_N(x/a)R_N(x/b)dx\\
&+\int_{1}^X U_N(x/b)R_N(x/a)dx+\int_{1}^X R_N(x/a)R_N(x/b)dx\\
&=\int_{1}^X U_N(x/a)U_N(x/b)dx+O\left(X^{1+1/4+\epsilon}+\frac{X^{1+3/4+\epsilon}}{\sqrt{N}} +\frac{X^{2+\epsilon}}{N} \right),
\end{align*}
	where we used the Cauchy-Schwarz inequality and \eqref{Tong} in the last equality.  By making the change of variable $u=x/\lambda$, we reach
\begin{align*}
&\int_{1}^X U_N(x/a)U_N(x/b)dx=\lambda\int_{1}^{X/\lambda} U_N(x/c)U_N(x/d)dx\\
&=\frac{\lambda}{2\pi^2(cd)^{1/4}}\sum_{n,m\leq N}\frac{\tau(n)\tau(m)}{(nm)^{3/4}}\int_{1}^{X/\lambda} x^{1/2}\cos(4\pi\sqrt{nx/c}-\pi/4)\cos(4\pi\sqrt{mx/d}-\pi/4)dx \\
&=\frac{\lambda}{\pi^2(cd)^{1/4}}
\sum_{n,m\leq N}
\frac{\tau(n)\tau(m)}{(nm)^{3/4}}
\int_{1}^{(X/\lambda)^{1/2}} u^2 
\cos(4\pi u\sqrt{n/c}-\pi/4)\cos(4\pi u\sqrt{m/d}-\pi/4)du,
\end{align*}
where in the last equality above we made a change of variable $u=\sqrt{x}$. We claim now that the main contribution comes when 
$n/c=m/d$. Since $c$ and $d$ are coprime, this implies that
$m=dk$ and $n=ck$. Therefore
the sum over these $n$ and $m$ can be written as
\begin{equation}
\label{equation diagonal terms}
\frac{\lambda}{\pi^2cd}
\sum_{k=1}^\infty
\frac{\tau(ck)\tau(dk)}{k^{3/2}}
\int_{1}^{(X/\lambda)^{1/2}} u^2\cos^2(4\pi u\sqrt{k}-\pi/4) du+O\left(\frac{X^{3/2+\epsilon}}{\sqrt{N}}\right).
\end{equation}
We recall now that $\cos^2(v)=\frac{1+\cos(2v)}{2}$, 
hence, by Lemma \ref{lema integrais} the integral above is
\begin{equation}\label{equation integral cos^2}
\int_{1}^{X^{1/2}/\lambda^{1/2}} x^2\cos^2(4\pi\sqrt{n}x-\pi/4)dx=\frac{X^{3/2}}{6\lambda^{3/2}}+O(X),
\end{equation}
where the big-oh term is uniform in $n$. 
Now we will show that the sum over those $n$ and $m$ such that $n/c\neq m/d$ will be $o(X^{3/2})$. 
With this the proof is complete by combining \eqref{equation diagonal terms} and \eqref{equation integral cos^2}. 

We recall the identity $2\cos(u)\cos(v)=\cos(u-v)+\cos(u+v)$.
Thus, for $\sqrt{n/c}\neq \sqrt{m/d}$, by Lemma \ref{lema integrais} we find that
\begin{align*}
&\int_{1}^{X^{1/2}/\lambda^{1/2}} x^2\cos(4\pi\sqrt{n/c}x-\pi/4)\cos(4\pi\sqrt{m/d}x-\pi/4)dx\\
&=\frac{1}{2}\int_{1}^{X^{1/2}/\lambda^{1/2}}  x^2\cos(4\pi(\sqrt{n/c}-\sqrt{m/d})x)dx+\frac{1}{2}\int_{1}^{X^{1/2}/\lambda^{1/2}} x^2\sin(4\pi(\sqrt{n/c}+\sqrt{m/d})x)dx\\
&\ll \frac{X}{\left|\sqrt{n/c}-\sqrt{m/d}\right|} + \frac{X}{\sqrt{n/c}+\sqrt{m/d}} \\
&\ll\frac{\sqrt{n/c}+\sqrt{m/d}}{|nd-mc|}X.
\end{align*}
Let $\ind_P(n)$ be the indicator that $n$ has property $P$.
We find that
\begin{align*}
&\sum_{\substack{n,m\leq N\\ nd-mc\neq0}}\frac{\tau(n)\tau(m)}{(nm)^{3/4}}\int_{1}^{X/\lambda} x^{1/2}\cos(4\pi\sqrt{nx/c}-\pi/4)\cos(4\pi\sqrt{mx/d}-\pi/4)dx\\
&\ll X N^\epsilon \sum_{\substack{n,m\leq N\\ nd-mc\neq0}}\frac{\sqrt{n/c}+\sqrt{m/d}}{(nm)^{3/4}|nd-mc|}\\
&=X N^\epsilon \sum_{\substack{n,m\leq N\\
  nd-mc\neq0}}\frac{\sqrt{n/c}+\sqrt{m/d}}{(nm)^{3/4}|nd-mc|}
  \sum_{\substack{k=-N\max(c,d)\\ k\neq 0}}^{N\max(c,d)} \ind_{nd-mc=k}.
\end{align*}
On calling this sum $S$, we readily continue with  
\begin{align*}
S&\ll X N^\epsilon \sum_{k=1}^{N\max(c,d)}\frac{1}{k}\sum_{m\leq N}\frac{\sqrt{m}+\sqrt{k}}{((k+mc)m)^{3/4}}\\
&\ll X N^\epsilon\left(O(\log N)^2+\sum_{k\leq N}\frac{1}{\sqrt{k}}\sum_{m\leq N}\frac{1}{(m^2+mk)^{3/4}} \right)\\
&\ll X N^\epsilon\left(O(\log N)^2+\sum_{k\leq N}\frac{1}{\sqrt{k}}\left(\sum_{k\leq m\leq N}\frac{1}{m^{3/2}}+\frac{1}{k^{3/4}}\sum_{m\leq k}\frac{1}{m^{3/4}}\right) \right)\\
&\ll X N^\epsilon (\log N)^2.  
\end{align*}
Finally, by selecting $N=X^2$, we arrive at
\begin{equation*}
    \int_1^X\Delta(x/a)\Delta(x/b)dx
    =
    \frac{1}{6\pi^2\sqrt{\lambda}cd}\left(
    \sum_{n=1}^\infty\frac{\tau(cn)\tau(dn)}{n^{3/2}}\right)X^{3/2}+O(X^{3/2-1/4+\epsilon}),
  \end{equation*}
  where the main contribution in the $O$-term above comes from the
  usage of Cauchy-Schwarz in the beginning of the proof.

The proof is complete.
\end{proof}
Now we deviate from the main line and prove Proposition \ref{proposition decorrelation}.
\begin{proof}[Proof of Proposition \ref{proposition decorrelation}]
By the proof of Lemma \ref{firstlimit} we have that
\begin{align*}
I_\theta(X):=&\int_{1}^X\Delta(x)\Delta(\theta x)dx\\
=&\frac{1}{\pi^2}
\sum_{n,m\leq N}
\frac{\tau(n)\tau(m)}{(nm)^{3/4}}
\int_{1}^{X^{1/2}} x^2 
\cos(4\pi x\sqrt{n}-\pi/4)\cos(4\pi x\sqrt{m\theta}-\pi/4)dx\\
&+O\left(X^{1+1/4+\epsilon}+\frac{X^{1+3/4+\epsilon}}{\sqrt{N}}+\frac{X^{2+\epsilon}}{N}\right).
\end{align*}
Now, by appealing to the identity $2\cos(u)\cos(v)=\cos(u-v)+\cos(u+v)$, we reach at
\begin{align*}
I_\theta(X)=\frac{1}{2\pi^2}
\sum_{n,m\leq N}
\frac{\tau(n)\tau(m)}{(nm)^{3/4}}
\int_{1}^{X^{1/2}} x^2 
\left(\cos(4\pi x(\sqrt{n}-\sqrt{m\theta}))  + \sin(4\pi x(\sqrt{n}+\sqrt{m\theta}))\right) dx\\+O\left(X^{1+1/4}+\frac{X^{1+3/4+\epsilon}}{\sqrt{N}} +\frac{X^{2+\epsilon}}{N}\right).
\end{align*}
We have that, by Lemma \ref{lema integrais},
\begin{align*}
\sum_{n,m\leq N}
\frac{\tau(n)\tau(m)}{(nm)^{3/4}}
\int_{1}^{X^{1/2}} x^2 
\sin(4\pi x(\sqrt{n}+\sqrt{m\theta}))dx & \ll X N^\epsilon \sum_{n,m\leq N}\frac{1}{m^{3/4}n^{5/4}+n^{3/4}m^{5/4}}\\
&\ll X N^\epsilon \sum_{n,m\leq N}\frac{1}{n^{3/4}m^{5/4}}\\
& \ll X N^{1/4+\epsilon}.
\end{align*}
 Thus, we reach at
\begin{align*}
I_\theta(X)=\frac{1}{2\pi^2}
\sum_{n,m\leq N}
\frac{\tau(n)\tau(m)}{(nm)^{3/4}}
\int_{1}^{X^{1/2}} x^2 
\cos(4\pi x(\sqrt{n}-\sqrt{m\theta}))dx\\+O\left(X^{1+1/4}+\frac{X^{1+3/4+\epsilon}}{\sqrt{N}} +\frac{X^{2+\epsilon}}{N}+XN^{1/4+\epsilon}\right).
\end{align*}

On calling the last sum above $S_\theta(X)$, $a_{n,m}:=4\pi(\sqrt{n} - \sqrt{m\theta})$, we obtain that
$$S_\theta(X)=X^{3/2}\sum_{n,m\leq N}
\frac{\tau(n)\tau(m)}{(nm)^{3/4}}\Lambda(a_{n,m}\sqrt{X}),$$
where, by Lemma \ref{lema integrais}, $\Lambda(0):=1/3$ and for $u\neq 0$ 
$$\Lambda(u):=\frac{\sin(u)}{u}+2\frac{\cos(u)}{u^2}-2\frac{\sin(u)}{u^3}.$$
A careful inspection shows that $\Lambda$ is continuous and for large $|u|$, $\Lambda(u)\ll |u|^{-1}$.

Now, for a large parameter $T$ (to be chosen), we split
$$S_\theta(X)=X^{3/2}\sum_{\substack{n,m\leq N\\ |a_{n,m}\sqrt{X}|\leq T}}
\frac{\tau(n)\tau(m)}{(nm)^{3/4}}\Lambda(a_{n,m}\sqrt{X})+X^{3/2}\sum_{\substack{n,m\leq
    N\\ |a_{n,m}\sqrt{X}|> T}}
\frac{\tau(n)\tau(m)}{(nm)^{3/4}}\Lambda(a_{n,m}\sqrt{X}).$$ We call
the first sum in the right hand side above \textit{diagonal}
contribution and the second sum the \textit{non-diagonal}
contribution. We select $T=X^{1/2-\delta}$ and $N=X^{1/2+\delta}$, for
some small $\delta>0$.

\noindent \textit{The diagonal contribution}. We have that
\begin{align}\label{equation diagonal contribution}
D(X)&=X^{3/2}\sum_{\substack{n,m\leq N\\ |a_{n,m}\sqrt{X}|\leq T}}
\frac{\tau(n)\tau(m)}{(nm)^{3/4}}\Lambda(a_{n,m}\sqrt{X})\\
&\ll X^{3/2}N^\epsilon\sum_{m\leq N}\frac{1}{m^{3/4}}\sum_{\substack{n;\\|n-m\theta|\leq\frac{2\sqrt{m\theta}}{X^\delta}+\frac{1}{X^{2\delta}}}}\frac{|\Lambda(a_{n,m}\sqrt{X})|}{n^{3/4}}.
\end{align}
The inner sum above we split accordingly $\frac{2\sqrt{m\theta}}{X^\delta}+\frac{1}{X^{2\delta}}$ is below and above $1$. In the case that this quantity is greater or equal to $1$, we have that $m\geq (4\theta)^{-1}X^{2\delta}$, and hence

\begin{align*}
  D(X)\ll
  &X^{3/2}N^\epsilon\sum_{(4\theta)^{-1}X^{2\delta}\leq m\leq N}\frac{1}{m^{3/4}}\sum_{\substack{n;\\|n-m\theta|\leq\frac{2\sqrt{m\theta}}{X^\delta}+\frac{1}{X^{2\delta}}}}\frac{|\Lambda(a_{n,m}\sqrt{X})|}{n^{3/4}}\\
&\ll
X^{3/2}N^\epsilon\sum_{ (4\theta)^{-1}X^{2\delta}\leq m\leq N}\frac{1}{m^{3/4}}\cdot \frac{1}{m^{3/4}}\frac{\sqrt{m}}{X^\delta}\\
&\ll X^{3/2-\delta}N^{\epsilon}.
\end{align*}
In the case that
$\frac{2\sqrt{m\theta}}{X^\delta}+\frac{1}{X^{2\delta}}\leq 1$, we
have that $m\leq(4\theta)^{-1}X^{2\delta}$, and now the
Diophantine properties of $\theta$ come in to play. If the
irrationality measure of $\theta$ is $\eta+1$, we have that for each $\epsilon$
there is a constant $C>0$ such that the inequality
$$|n-m\theta|\geq \frac{C}{m^{\eta+\epsilon}}$$ 
is violated only for a finite number of positive integers $n$ and
$m$. In our case, this allows us to lower bound $|a_{n,m}\sqrt{X}|$
for all but a finite number of $n$ and $m$ such that
$1\leq m\ll X^{2\delta}$ and $1/2\le \sqrt{n}/\sqrt{m\theta}\le 2$:
\begin{align*}
|a_{n,m}\sqrt{X}|\cdot\frac{\sqrt{n}+\sqrt{m\theta}}{\sqrt{n}+\sqrt{m\theta}}&=\sqrt{X}\frac{|n-m\theta|}{\sqrt{n}+\sqrt{m\theta}}\\
&\geq \frac{\sqrt{X}}{m^{\eta+\epsilon}(\sqrt{n}+\sqrt{m\theta})}\\
&\gg X^{1/2-(2\eta+1)\delta-\epsilon}.  
\end{align*}
Observe that the diagonal contribution from those exceptional $n$ and $m$ will be at most $O(X)$. With these estimates on hand and recalling that $\Lambda(u)\ll |u|^{-1}$, we obtain

\begin{align*}
&X^{3/2}N^\epsilon\sum_{m\leq (4\theta)^{-1}X^{2\delta}}\frac{1}{m^{3/4}}\sum_{\substack{n;\\|n-m\theta|\leq\frac{2\sqrt{m\theta}}{X^\delta}+\frac{1}{X^{2\delta}}}}\frac{|\Lambda(a_{n,m}\sqrt{X})|}{n^{3/4}}\\
&\ll X^{3/2}N^\epsilon\sum_{m\leq(4\theta)^{-1}X^{2\delta}}\frac{1}{m^{3/2}}\cdot \frac{1}{X^{1/2-(2\eta+1)\delta-\epsilon}}+O(X)\\
& \ll X^{1+(2\eta+1)\delta+\epsilon}.
\end{align*}
Therefore, the diagonal contribution is at most
$$D(X)\ll X^{1+(2\eta+1)\delta+\epsilon} +X^{3/2-\delta+\epsilon}.$$

\noindent \textit{The non-diagonal contribution}. Now, we reach
\begin{align*}
X^{3/2}\sum_{\substack{n,m\leq N\\ |a_{n,m}\sqrt{X}|> T}}
\frac{\tau(n)\tau(m)}{(nm)^{3/4}}\Lambda(a_{n,m}\sqrt{X})&\ll \frac{X^{3/2}N^{1/2+\epsilon}}{T}\\
&=X^{3/2+1/4+(\delta+\epsilon)/2+\epsilon\delta-1/2+\delta}\\
&=X^{1+1/4+3\delta/2+\epsilon/2+\epsilon\delta}.
\end{align*}
We choose $\delta=\frac{1}{3(2\eta+1)}$ and obtain
$$I_\theta(X)=O(X^{3/2-1/(18\eta)}).$$
The proof of the first part of Proposition \ref{proposition decorrelation} is complete. 

Now we assume that $\theta$ is a Liouville number, \textit{i.e.},
$\theta$ does not have finite irrationality measure. We see that the
non-diagonal argument does not depend on the Diophantine properties of
$\theta$. Let $\eta>0$ be a large fixed number, $t>0$ a small number
that will tend to $0$. For $D(X)$ as in \eqref{equation diagonal
  contribution}, by repeating verbatim the estimates above we have
that
$$D(X)\ll X^{3/2}\sum_{m\leq(4\theta)^{-1}X^{2\delta}}\frac{\tau(m)}{m^{3/4}}\sum_{\substack{n;\\|n-m\theta|\leq\frac{2\sqrt{m\theta}}{X^\delta}+\frac{1}{X^{2\delta}}}}\frac{\tau(n)|\Lambda(a_{n,m}\sqrt{X})|}{n^{3/4}}+O(X^{3/2-\delta}N^{\epsilon}).$$
Let $\|x\|$ be the distance from $x$ to the nearest integer. We split
the sum over $m$ above into two sums: One over those $m$ such that
$\|m\theta\|> t m^{-\eta}$ and the other over $m$ such that
$\|m\theta\|\leq t m^{-\eta}$.

Repeating the argument above for non-Liouville numbers, we have that the contribution over those $m$ such that $\|m\theta\|> t m^{-\eta}$ is $O(t^{-1}X^{1+\delta(2\eta+1)})$. Therefore
$$D(X)\ll X^{3/2}\sum_{\substack{m=1\\{\|m\theta\|\leq tm^{-\eta}}}}^\infty\frac{1}{m^{3/2-\epsilon}}+O(t^{-1}X^{1+\delta(2\eta+1)}+X^{3/2-\delta+\epsilon}).$$
Combining all these estimates, we see that
$$\limsup_{X\to\infty}\frac{1}{X^{3/2}}\left|\int_{1}^X\Delta(x)\Delta(\theta x)dx\right| \ll \sum_{\substack{m=1\\{\|m\theta\|\leq tm^{-\eta}}}}^\infty\frac{1}{m^{3/2-\epsilon}}.$$
Since the upper bound above holds for all $t>0$, we have that as
$t\to0^+$, the sum above converges to $0$ and thus implying that the
$\limsup$ is $0$. The proof is complete.  \end{proof}

\begin{proof}[Proof of Theorem~\ref{TongGen}]
  On combining Lemma~\ref{firstlimit} together with
  Lemma~\ref{exactcomp}, we get the first part of
  Theorem~\ref{TongGen}. The second part is a trivial consequence of
  Proposition~\ref{proposition decorrelation}.
\end{proof}

\section{Quadratic forms auxiliaries}\label{section quadratic}

The main proof will lead to considering the quadratic form attached
to a matrix of the form
\begin{equation}
  \label{mymat}
  M_{S, \varphi}=\biggl(\frac{1}{\sqrt{\gcd(a,b)}}
  \varphi\biggl(\frac{\lcm(a,b)}{\gcd(a,b)}\biggr)
  \biggr)_{a,b\in S}
\end{equation}
where $S$ is some finite set of integers while $\varphi$ is a
\emph{non-negative multiplicative function such that $\varphi(p^k)\le 1$}.
So we stray somewhat from the main line and investigate this
situation.  Our initial aim is to find conditions under which the
associated quadratic form is positive definite, but we shall finally
restrict our scope.  GCD-matrices have received quite some attention,
but it seems the matrices occuring in~\eqref{mymat} have not been
explored.  We obtain results in two specific contexts.
\subsubsection*{Completely multiplicative case}
Here is our first result.
\begin{lemma}
  \label{CM}
  When $\varphi$ is completely multiplicative and $p^{1/4}\varphi(p)\in(0,1]$, the matrix $M_{S, \varphi}$ is
  non-negative. When in addition we assume that $p^{1/4}\varphi(p)\in(0,1)$ and $S$ is divisor closed, this matrix
  is positive definite. The determinant in that case is given by the
  formula
  \begin{equation*}
    \det
    \biggl(\frac{1}{\sqrt{\gcd(a,b)}}
    \varphi\biggl(\frac{\lcm(a,b)}{\gcd(a,b)}\biggr)
    \biggr)_{a,b\in S}
    =\prod_{d\in S}\varphi(d)^2(\mu\ast \psi)(d),
  \end{equation*}
  where $\psi$ is the completely multiplicative function given by $\psi(p)=1/(\sqrt{p}\varphi(p)^2)$.
\end{lemma}
By \emph{divisor closed}, we mean that every divisor of an element of
$S$ also belongs to~$S$.

\begin{proof}
  We write
  \begin{equation*}
  \frac{1}{\sqrt{\gcd(a,b)}}
    \varphi\biggl(\frac{\lcm(a,b)}{\gcd(a,b)}\biggr)
    =\varphi(a)\varphi(b)\psi(\gcd(a,b))
  \end{equation*}
  where $\psi(n)=1/(\varphi(n)^2\sqrt{n})$ 
  is another non-negative multiplicative
  function. We introduce the auxiliary function $h=\mu\ast
  \psi$. Notice that this function is multiplicative and
  non-negative, as $\psi(p)\ge 1$. We use Selberg's diagonalization
  process  to write
  \begin{align*}
    \sum_{a,b\in S}
    \frac{1}{\sqrt{\gcd(a,b)}}
    \varphi\biggl(\frac{\lcm(a,b)}{\gcd(a,b)}\biggr)x_ax_b
    &=
      \sum_{a,b\in S}
      \psi(\gcd(a,b)) \varphi(a)x_a\varphi(b)x_b
     \\& =\sum_{a,b\in S}
      \sum_{d|\gcd(a,b)}h(d)\varphi(a)x_a\varphi(b)x_b\\
    &=
    \sum_{d} h(d)\biggl(\sum_{\substack{a\in S\\ d|a}}\varphi(a)x_a\biggr)^2
  \end{align*}
  from which the non-negativity follows readily. When $\varphi$ verifies
  the more stringent condition that $p^{1/4}\varphi(p)\in(0,1)$, we know that both
  $\varphi$ and $h$ are strictly positive. Let us define
  $y_d=\sum_{\substack{a\in S\\ d|a}}\varphi(a)x_a$. The variable $d$ varies
  in the set $D$ of divisors of $S$. We assume that $S$ is divisor
  closed, so that $D=S$. We can readily invert
  the triangular system giving the $y_d$'s as functions of the $x_a$'s into
  \begin{equation*}
    \varphi(a)x_a=\sum_{a|b}\mu(b/a)y_b
  \end{equation*}
  Indeed, the fact that the mentioned system is triangular ensures
  that a solution $y$ is unique if it exists. We next verify that the proposed expression
  is indeed a solution by:
  \begin{equation*}
    \sum_{\substack{a\in S\\ d|a}}\varphi(a)x_a
    =\sum_{\substack{a\in S\\ d|a}}\sum_{a|b}\mu(b/a)y_b
    =\sum_{\substack{b\in S\\ d|b}}y_b\sum_{d|a|b}\mu(b/a)=y_d
  \end{equation*}
  as the last inner sum vanishes when $d\neq b$. We thus have a
  writing as a linear combination of squares of independant linear
  forms. In a more pedestrian manner, if our quadratic form vanishes,
  then all $y_d$'s do vanish, hence so do the $x_a$'s.
\end{proof}
Here is a corollary.
\begin{lemma}
  \label{CMb}
  When the set $S$ contains solely squarefree integers, the matrix
  $M_{S,\varphi}$ is non-negative.
\end{lemma}

\begin{proof}
  Simply apply Lemma \ref{CM} to the completely multiplicative function $\varphi'$ that
  has the same values on primes as $\varphi$. 
\end{proof}


Now we recall the Sylvester's criterion.

\begin{lemma}\label{Additive}
  A hermitian complex valued matrix $M=(m_{i,j})_{i,j\le K}$ defines a
  positive definite form if and only if
  all its principal minors
 $
   \det(m_{i,j})_{i,j\le k}
  $ for $k\le K$
  are positive.
\end{lemma}
\subsubsection*{A tensor product-like situation}
Lemma~\ref{CMb} is enough to solve our main problem when $M_1$ and
$M_2$ are coprime squarefree integers. We need to go somewhat further. Let $S$ be a
divisor closed set. We consider the quadratic form
\begin{equation}
  \label{Q1}
  \sum_{a,b\in S}\varphi\biggl(\frac{\lcm(a,b)}{\gcd(a,b)}\biggr)x_a x_b
\end{equation}
where the variables $x_a$'s are also multiplicatively split, i.e.
\begin{equation}
  \label{V1}
  x_a=\prod_{p^k\|a}x_{p^k}.
\end{equation}
Let $S(p)$ the subset of $S$ made only of $1$ and of prime powers. We
extend $S$ so that it contains every products of integers from
any collection of distinct $S(p)$\footnote{This is not automatically
  the case, as the example $S=\{1,2,3,5,6,10\}$ shows, since $30$ does
not belong to~$S$}.
We then find that
\begin{equation}
  \label{Qfactored}
  \sum_{a,b\in S}\frac{1}{\sqrt{\gcd(a,b)}}
  \varphi\biggl(\frac{\lcm(a,b)}{\gcd(a,b)}\biggr)x_a x_b
  =\prod_{p\in S}\biggl(
  \sum_{p^{k},p^\ell\in S(p)}
  \frac{\varphi\bigl(p^{\max(k,\ell)-\min(k,\ell)}\bigr)}{p^{\min(k,\ell)/2}}
  x_{p^k} x_{p^\ell}
  \biggr).
\end{equation}
We check this identity simply by opening the right-hand side and
seeing that every summand from the left-hand side appears one and only
one time. 


\section{Proof of the main result}
\begin{proof}
  By \cite[Theorem 1.4]{Aymone*22}, we have
  \begin{equation*}
      S(x)=\sum_{n\le x}(f_1\ast f_2)(n)
      =
      \sum_{a| M_1M_2}g(a)\Delta(x/a)
  \end{equation*}
  where $g=f_1\ast f_2\ast\mu\ast\mu$. We infer from this formula that
  \begin{align*}
      &\int_1^X|S(x)|^2dx
      =
      \sum_{a,b| M_1M_2}g(a)g(b)
      \int_1^X \Delta(x/a)\Delta(x/b)dx
      \\&=
      \frac{(1+o(1))}{6\pi^2}X^{3/2}
      \sum_{a,b|M_1M_2}g(a)g(b)
      \frac{\gcd(a,b)^{3/2}}{ab}
      \sum_{n= 1}^\infty
      \frac{\tau(an/\gcd(a,b))\tau(bn/\gcd(a,b))}{n^{3/2}}
  \end{align*}
  by Lemma~\ref{firstlimit}. We next use Lemma~\ref{exactcomp} to infer that
  \begin{equation*}
      \lim_{X\to\infty}
      \frac{1}{X^{3/2}}\int_1^X|S(x)|^2dx
      =
      \frac{\zeta(3/2)^4}{6\pi^2\zeta(3)}
      \sum_{a,b| M_1M_2}g(a)g(b)
      \frac{1}{ \sqrt{\gcd(a,b)}}
      \varphi\biggl(\frac{\lcm(a,b)}{\gcd(a,b)}\biggr)
  \end{equation*}
  where $\varphi$ is multiplicative and at prime powers: 
\begin{equation}
\begin{split}
\label{fval}
\varphi(p^k) 
&= \frac{(k+1)}{p^{k}}\frac{1}{1+p^{-3/2}}
      \biggl(1-\frac{(k-1)}{(k+1)p^{3/2}}\biggr)\\ 
&= \frac{1}{p^{k} (1 + p^{-3/2})} \biggl(  (k+1) - (k-1)p^{-3/2} \biggr)    \\
&= \frac{1}{p^{k} (1 + p^{-3/2})} \biggl( k (1-p^{-3/2}) + (1 + p^{-3/2}) \biggr)\\
&=\frac{k \bt(p) + 1 }{p^{k}},
\end{split}
\end{equation}
where
$$
 \bt(p) = \frac{1-p^{-3/2}}{1 + p^{-3/2}}.
$$

Now, we can write
\begin{equation*}
\frac{1}{\sqrt{\gcd(a,b)}} \varphi\left( \frac{\lcm(a,b)}{\gcd(a,b)} \right)  = \frac{1}{(ab)^{1/4}} \left( \frac{\lcm(a,b)}{\gcd(a,b)} \right)^{1/4} \varphi\left( \frac{\lcm(a,b)}{\gcd(a,b)} \right). 
\end{equation*}

Since the terms $a^{-1/4}$ and $b^{-1/4}$ can be absorbed into the variables $g(a)$ and $g(b)$ of the quadratic form, it is enough to consider the quantity  

$$
\varphi^*\left( \frac{\lcm(a,b)}{\gcd(a,b)}  \right), \quad \text{where} \ \, \varphi^*(n) = n^{1/4} \varphi(n) .
$$

We note that, at the prime power $p^k$, we have
\begin{equation}
\label{gval}
\varphi^*(p^k) = p^{k/4} \varphi(p^k) =\frac{k \bt(p) + 1 }{p^{3k/4}} . 
\end{equation}

Due to \eqref{Qfactored} and the discussion before it, we now restrict to the prime power case, that is, we look to matrices of the form 

\begin{equation*}
\Mcal_{K} = \bigl(  \varphi^*(p^{|i-j|}) \bigr)_{i,j \leq K}.  
\end{equation*}
As we are dealing with a given prime $p$, we shorten $\beta(p)$ in $\bt$.

\medskip

Since $\varphi^*$ is not completely multiplicative, it is not clear how to handle the matrix $\Mcal_K$ directly. So, our aim will be to transform this into another matrix which, in some way associates with a completely multiplicative function.  
So, let us consider

$$
\A_K = \U_K^{\top} \Mcal_K \, \U_K,
$$
where,
\begin{equation}
\label{Uij}
\U_K(i,j) = \begin{dcases}  \frac{\mu(p^{|i-j|})}{p^{3(|i-j|)/4}}, & \text{when } i \geq j \, \text{ or } \, (i,j) = (K-1, K),  \\0, & \text{otherwise.}  \end{dcases}
\end{equation}

Simply speaking, $\U_K$ is $1$ on the diagonal and $-p^{-3/4}$ on all $(i+1,i)$ as well as $(K-1, K)$. 
 Also
$$\det(\U_K) = 1 - p^{-3/2} .$$ 

We now calculate the entries of the matrix $\A_K$. We have the following:

\begin{proposition}
\label{proposition A}
The matrix $\A_K$ above is given by:
$$
\A_K(i,j) = \bt(1 - p^{-3/2}) \cdot \begin{dcases} p^{-3|i-j|/4}, & \text{when } \, {1 \leq i, j \leq K-1} \text{ or } i=j=K, \\ 0, & \text{otherwise}. \end{dcases}
$$
\end{proposition}

\medskip

We begin with the following lemma:
\begin{lemma}
\label{gpm}
We have
$$
\varphi^*(p^m) - p^{-3/4} \varphi^*(p^{|m-1|}) = p^{-3m/4} \bt, \quad \text{for all } \, m \geq 0. 
$$
\end{lemma}
\begin{proof}
First, assume $m \geq 1$. We have
\begin{equation*}
\varphi^*(p^m) - p^{-3/4} \varphi^*(p^{m-1}) = \frac{ m\bt + 1}{p^{3m/4}} - p^{-3/4} \frac{(m-1)\bt + 1}{p^{3(m-1)/4}} = p^{-3m/4} \bt .
\end{equation*}
When $m=0$, we have
\begin{equation*}
1 - p^{-3/4} \varphi^*(p) = 1 - p^{-3/2}(\bt+1) = 1 - \frac{2p^{-3/2}}{1 + p^{-3/2}} = \bt. 
\end{equation*}
\end{proof}
Now, we shall proceed with the proof of the Proposition \ref{proposition A}.
\begin{proof}[Proof of Proposition \ref{proposition A}]

Let us first assume that $1 \leq i, j \leq K-1$. We have 
\begin{equation}
\begin{split}
\label{Aij}
\A_K(i,j) &= \sum\limits_{k_1, k_2} \U_K^{\top}(i, k_1) \Mcal_K(k_1, k_2) \,\U_K(k_2, j)\\
&= \sum\limits_{\substack{k_1 - i \in \{0,1\} \\ k_2 - j \in \{0,1\} } } \frac{\mu(p^{k_1 - i})}{p^{3(k_1-i)/4}} \frac{\mu(p^{k_2 - j})}{p^{3(k_2-j)/4}} \varphi^*(p^{|k_1 - k_2|})\\
&= \biggl( \varphi^*(p^{|i-j|}) \bigl( 1 + p^{-3/2} \bigr) - \frac{\varphi^*(p^{|i-j+1|}) + \varphi^*(p^{|i-j-1|}) }{p^{3/4}} \biggr).
\end{split}
\end{equation}


Here, we do not have the contribution coming from $\U_K(K-1, K)$ or $\U_K^{\top} (K, K-1)$ as we have assumed $i, j \leq K-1$. This assumption is necessary because we are considering the values $k_1 = i+1$ and $k_2 = j+1$ (both of which should remain $\leq K$). 

\medskip

First, let us consider the case $i \geq j$. Letting $i-j = m \geq 0$, \eqref{Aij} becomes 
\begin{equation*}
\begin{split}
\A_K(i+m, i) &= \varphi^*(p^m)  - p^{-3/4} \varphi^*(p^{|m-1|}) - p^{-3/4} \bigl( \varphi^*(p^{m+1}) - p^{-3/4} \varphi^*(p^m) \bigr) \\
&= p^{-3m/4} \bt - p^{-3/4} p^{-3(m+1)/4} \bt \\
&= \bt(1 - p^{-3/2}) p^{-3m/4}.
\end{split}
\end{equation*}
Similarly, for $j \geq i$, we will obtain the same expression in terms of $m = j-i$. This proves Proposition \ref{proposition A} for $1 \leq i,j \leq K-1$. 

\medskip

Next, we consider the case when one of $i$ or $j$ equals $K$. 

\textbf{Claim}: $\A_K(i,K) = \A_K(K, j) = 0$, for all $1 \leq i, j \leq K-1$. 

We revert to the first line of the expression \eqref{Aij}. Letting $m = K-i \geq 1$, we obtain
\begin{equation*}
\begin{split}
\A_K(i, K) 
&= \sum\limits_{\substack{k_1 \in \{i, i+1\}\\ k_2 \in \{K-1, K\}}} \frac{\mu(p^{k_1 - i})}{p^{3(k_1 - i)/4}} \frac{\mu(p^{K-k_2})}{p^{3(K-k_2)/4}} \varphi^*(p^{|k_1 - k_2|})\\
&= - p^{-3/4} \varphi^*(p^{m-1})  + p^{-3/2} \varphi^*(p^{|m-2|}) + \varphi^*(p^m)  - p^{-3/4} \varphi^*(p^{m-1}) \\
&= -p^{-3/4} \big( \varphi^*(p^{m-1}) - p^{-3/4} \varphi^*(p^{|m-2|})  \bigr) + \varphi^*(p^m)  - p^{-3/4} \varphi^*(p^{m-1})\\
&= -p^{-3/4} p^{-3(m-1)/4} \bt  +  p^{-3m/4} \bt
= 0.
\end{split}
\end{equation*}

It similarly follows that $\A_K(K, j) = 0$ for $1 \leq j \leq K-1$, proving the claim. 

Next, we see that
\begin{equation*}
\begin{split}
\A_K(K, K) 
&= \sum\limits_{k_1, k_2 \in \{K-1, K\}} \frac{ \mu(p^{K-k_1})}{p^{3(K-k_1)/4}} \frac{\mu(p^{K-k_2})}{p^{3(K-k_2)/4}} \,\varphi^*(p^{|k_1 - k_2|}) \\
&= 1 - p^{-3/4} \varphi^*(p) -p^{-3/4} \bigl( \varphi^*(p) - p^{-3/4}  \bigr)\\
&= \bt - p^{-3/4} \bigl( p^{-3/4} \bt \bigr) 
= \bt(1 - p^{-3/2}).
\end{split}
\end{equation*}
This completes the proof of Proposition \ref{proposition A}. 
\end{proof}

Now since $n \mapsto n^{-3/4}$ is completely multiplicative, by the proof of Lemma \ref{CM}, the matrix 
$$\mathcal{B}_K=\left(\left( \frac{\lcm(a,b)}{\gcd(a,b)}\right)^{-3/4}\right)_{a,b\in\{1,...,p^K\}}$$
is positive definite for all $K$. Since the entries $(i,j)$ with $1\leq i,j\leq K-1$ of $\mathcal{A}_K$ coincide with the ones of the matrix $c\mathcal{B}_K$ for some positive constant $c$, and that the entries $(i,K)$ and $(K,j)$ of $\mathcal{A}_K$ are all zero with a single exception at the entry $(K,K)$, by Sylvester's criterion (Lemma \ref{Additive}), we conclude that the matrix $\mathcal{A}_K$ is positive definite for all $K$.

Since $\mathcal{A}_K= \U_K^{\top} \Mcal_K \U_K$, we have
$$
\det(\A_K) = \det(\U_K)^2 \, \det(\Mcal_K) = (1-p^{-3/2})^2 \, \det(\Mcal_K). 
$$
This proves that $\det(\Mcal_K)>0$ and by induction over $K$ in Lemma \ref{Additive}, $\Mcal_K$ is positive definite for all $K$. 

The factorization \eqref{Qfactored} completes the proof of Theorem \ref{theorem principal}.
\end{proof}
\noindent\textbf{Acknowledgements.} We would like to thank the referee for a careful revision and for his/her comments, suggestions and corrections that improved this paper. This project started while the
first author was a visiting Professor at Aix-Marseille Universit\'e,
and he is thankful for their kind hospitality and to CNPq grant PDE no. 400010/2022-4 (200121/2022-7) for supporting this visit. The first author also is
supported by FAPEMIG, grant Universal no. APQ-00256-23 and by CNPq
grant Universal no. 403037/2021-2. The second and third author are
supported by the joint FWF-ANR project Arithrand: FWF: I 4945-N and
ANR-20-CE91-0006. The main bulk of this paper was build when the
fourth author was a guest of the
Aix-Marseille Universit\'e I2M laboratory. We thank this place for its support.

\end{document}